\documentclass[12pt,reqno,tbtags]{amsart}
\usepackage{amsmath}
\usepackage{float}
\usepackage{graphicx}
\usepackage{amsthm}
\usepackage{amssymb}
\usepackage{framed}
\usepackage{tikz}
\usetikzlibrary{arrows}
\usepackage[margin=1in]{geometry}

\usepackage{caption,sansmath}
\DeclareCaptionFont{sansmath}{\sansmath}
\newtheorem{theorem}{Theorem}[section]
\newtheorem{lem}[theorem]{Lemma}

\newtheorem{definition}[theorem]{Definition}
\newtheorem{conjecture}{Conjecture}[section]

\newtheorem{alg}{Algorithm}[section]

\newcommand{\mc}[1]{\mathcal{#1}}

\newcommand{\eps}{\varepsilon}

\newcommand{\sub}{\subseteq}
\newcommand{\sm}{\setminus}
\newcommand{\es}{\emptyset}

\usepackage{amsaddr}

\begin{document}
\date{}

\author{Peter Keevash and Liana Yepremyan}
\address{Mathematical Institute, University of Oxford.}
\thanks{Research supported in part by ERC Consolidator Grant 647678.}
\email{\{keevash,yepremyan\}@maths.ox.ac.uk}

\title{Rainbow matchings in properly-coloured multigraphs}

\maketitle

\begin{abstract}
Aharoni and Berger conjectured that 
in any bipartite multigraph that is
properly edge-coloured by $n$ colours 
with at least $n + 1$ edges of each 
colour there must be a matching 
that uses each colour exactly once.
In this paper we consider the same question
without the bipartiteness assumption.
We show that in any multigraph 
with edge multiplicities $o(n)$
that is properly edge-coloured by $n$ colours with 
at least $n + o(n)$ edges of each colour 
there must be a matching of size $n-O(1)$
that uses each colour at most once.
\end{abstract}

\section{Introduction}
A Latin square of order $n$ is an $n\times n$ array of $n$ symbols in which each symbol occurs exactly once in each row and in each column.  
A \emph{transversal} of a Latin square is a set of entries
that represent each row, column and symbol exactly once.
The following fundamental open problem on transversals 
in Latin squares is known as the Ryser-Brualdi-Stein conjecture%
\footnote{Ryser conjectured that the number of transversals 
has the same parity as $n$, 
so any Latin square of odd order has a transversal
(see~\cite{wanless}).
For any even $n$ the addition table of a cyclic
group is a Latin square with no transversal.
Brualdi made Conjecture \ref{conj:brualdistein}
(see e.g.~\cite{brualdi}); Stein~\cite{stein} 
independently made a stronger conjecture.}
(the best known bound is $n-O(\log^2n)$
by Hatami and Shor~\cite{hatamishor}).
\begin{conjecture}[Ryser-Brualdi-Stein]\label{conj:brualdistein} 
Every Latin square of order $n$ has a partial transversal of size $n -1$.
\end{conjecture}

It is not hard to see that Conjecture \ref{conj:brualdistein}
is equivalent to saying that any proper $n$-edge-colouring
of the complete bipartite graph $K_{n,n}$ has a rainbow matching
(i.e.\ a matching with no repeated colour) of size $n-1$.
Aharoni and Berger~\cite{aharoniberger} conjectured 
the following generalization.

\begin{conjecture}[Aharoni and Berger]\label{conj:ab} 
Let $G$ be a bipartite multigraph that is properly edge-coloured
with $n$ colours and has at least $n+1$ edges of each colour. 
Then $G$ has a rainbow matching using every colour.
\end{conjecture}

Pokrovskiy~\cite{pokrovskiy} showed that this conjecture
is asymptotically true, in that the conclusion holds
if there are at least $n+o(n)$ edges of each colour. 
In this paper we consider the same question
without the bipartiteness assumption.
We obtain a result somewhat analogous to Pokrovskiy's,
although we require an additional (mild) assumption
on edge multiplicities, and we have to allow
a constant number of unused colours.
(We give explicit constants in the statement
for the sake of concreteness, but we did
not attempt to optimise these.)

\begin{theorem}\label{thm:main}
Let $0<\eps<10^{-3}$, $k = 2^{20/\eps}$ and $n>k^2$.
Suppose $G$ is any multigraph with 
edge multiplicities at most $n/k$
that is properly edge-coloured with $n$ colours so that
every colour appears at least $(1+\eps)n$ times.
Then $G$ has a rainbow matching of size $n-k$.
\end{theorem}

Our proof of Theorem~\ref{thm:main} is algorithmic.
Given any rainbow matching $M$ of size less than $n-k$,
we construct a hierarchy of edges $e$ in $M$ for which
we can find another rainbow matching $M'$ of the same size
which does not contain $e$, or indeed
any edge with the same colour as $e$.
We find $M'$ by a sequence of switches
using edges at lower levels in hierarchy.
The switching method is robust, in that there are many
choices at each step, and so it is possible to avoid
any constant size set of vertices or edges
that were altered by previous switches.
To analyse the algorithm, we suppose for contradiction
that it does not find a switch to increase the matching,
and then prove several structural properties of $G$
that culminate in a counting argument
to give the required contradiction.

We describe and analyse our algorithm 
in Section~\ref{sec:alg}, after briefly
summarising our notation in the next section.
Then we give the proof of Theorem~\ref{thm:main}
in Section~\ref{sec:mainproof}.
In the concluding remarks we discuss the relationship
of our work to the recent preprint of
Gao, Ramadurai, Wanless and Wormald \cite{grww}.

\section{Notation}
\label{sec:notation}
For convenient reference we summarize 
our notation in this section. 

Let $G$ be a properly edge-coloured multigraph.
We write $c(e)$ for the colour of an edge $e$.
We also say that $e$ is a $c$-edge, where $c(e)=c$.
Similarly, for any set $C$ of colours,
we say that $e$ is a $C$-edge if $c(e) \in C$.

Now suppose $M$ is a rainbow matching in $G$.
We let $C(M)$ denote the set of colours used by $M$.
For $c \in C(M)$ we let $m_c$ be the edge of $M$ of colour $c$.
If $e=uv\in M$ we say that $u$ and $v$ are {\em twins},
and write $u=t(v)$ (and so $v=t(u)$).
For $S \sub V(M)$ we write $T(S)=\{t(v): v \in S\}$.

Any edge that goes between $V(G) \setminus V(M)$ 
and $V(M)$ is called \emph{external}.

We say that $M$ \emph{avoids} a
set $S$ of vertices if $V(M)\cap S=\emptyset$. 
We say that $M$ \emph{avoids} a
set $C$ of colours if $C(M)\cap C=\emptyset$. 
We say that $M$ \emph{fixes} a 
set $E$ of edges if $E \sub M$.

Suppose $M$ and $M'$ are rainbow matchings in $G$.
We say $M$ and $M'$ are $\lambda$-close
if $|M'|=|M|$ and $|M\triangle M'|\leq \lambda$.
Note that this is non-standard terminology
(standard usage would not include the condition $|M'|=|M|$).

For the remainder of the paper 
we fix $\eps$, $k$, $n$ and $G$
as in the statement of Theorem~\ref{thm:main}. We also use an auxiliary variable $\alpha:=\eps/12$. We fix a rainbow matching $M$ in $G$ of maximum size
and write $V_0 = V(G) \sm V(M)$.
Let $C_0$ be the set of colours not used by $M$.
We note that by maximality there
are no $C_0$-edges contained in $V_0$.
We suppose for a contradiction that $|C_0|>k$.

\section{The reachability algorithm}
\label{sec:alg}

We will describe an algorithm which builds a set of colours that we call \emph{reachable}. Informally, these are the colours $c$ in $M$ such that we can (robustly) replace $c$ by some unused colour $c_0 \in C_0$ and obtain another rainbow matching $M'$ with $|M'|=|M|$. To achieve this we may need to make several other changes to the matching, but if the number of changes is bounded by some constant then we think of $c$ as reachable.
The changes will consist of several flips
along disjoint $M$-alternating paths
(as in the classical matching algorithms).

As a first step, let us show that there are many edges in $M$
which we call `flexible', meaning they have several options 
to be directly replaced by a $C_0$-edge,
with no further changes to the matching.
To be precise, we let $E_0$ be the set of edges in $M$
which can be oriented from tail to head so that
the number of external $C_0$-edges incident to the tail
is at least $\alpha |C_0|$ (if both orientations
work then we choose one arbitrarily).
We let $S_0$ denote the set of heads
of edges in $E_0$ (so the set of tails
is the set $T(S_0)$ of twins of $S_0$).
Then $F=C(E_0)$ is the set of {\em flexible} colours.

For any $v \in S_0$ there are at least
$\alpha|C_0|$ choices of a $C_0$-edge $t(v)u$ 
with $u \in V_0$, and $M' = M \sm \{vt(v)\} \cup \{t(v)u\}$
is a rainbow matching with $|M'|=|M|$. 
However, it may be that $u$ is the same 
for all such $C_0$-edges, so we do not consider
flexible edges to be `reachable'; for this we
will require linearly many options for a switch.
Nevertheless, for the analysis of the algorithm
we need many flexible edges.

\begin{lem}\label{lem:f}
$|F|\geq \alpha n$.
\end{lem}
\begin{proof}
For every $c \in C_0$, we note that there are at least $\eps n$
external $c$-edges. Indeed, the total number of $c$-edges is
at least $n + \eps n$, there are none in $V_0$, 
and at most $n$ are contained in $V(M)$, as $|M|<n$
and the colouring is proper. Thus we have
at least $\eps n \cdot |C_0|$ external $C_0$-edges.
Note that any vertex in $V(M)$ is incident
to at most $|C_0|$ of these edges,
and by definition of $F$, any vertex
of $V(M) \sm (S_0 \cup T(S_0))$
is incident to at most $\alpha |C_0|$ of them.
Thus $\eps n \cdot |C_0| \le  2|F| \cdot |C_0| + 2n \cdot \alpha |C_0|$,
so $|F| \geq \left( \frac{\eps}{2}-\alpha\right)n \geq \alpha n$.
\end{proof}

For more options of switches we need longer paths.
For example, if an edge $e$ of $M$ is incident to many 
external $F$-edges, we may be able to use one of these 
edges instead of $e$, and then remove the corresponding 
$F$-edge from $M$ using its incident $C_0$-edges.
We will be able to implement this idea 
for edges satisfying the following condition.
For $c\in F$, we classify external $c$-edges
as \emph{good} or \emph{bad}, where 
$e$ is good if $m_c$ (the $c$-edge in $M$)
is incident to at least $\alpha |C_0|/2$ external 
$C_0$-edges which share no endpoint with $e$.

Now we formulate our algorithm.

\begin{alg}
We iteratively define sets $E_i$ 
of `$i$-reachable' edges.
Each edge in $E_i$ will be assigned
an orientation, from tail to head.
We let $V_i$ denote the set of heads
of edges in $E_i$ (so the set of tails
is the set $T(V_i)$ of twins of $V_i$).
We write $R_i=C(E_i)$ for the 
set of `$i$-reachable' colours.

We let $E_1$ be the set of edges in $M$
that can be oriented from tail to head such that the tail is incident 
to at least $\alpha |F|$ good $F$-edges. 

For $i>1$ we let $E_i$ be the set of
edges in $M \sm \cup_{j<i} E_j$
that can be oriented from tail $t$ to head such that for some $j<i$
the number of $R_j$-edges $tu$
with $u \in \cup_{j=0}^{i-1} V_j$
is at least $\alpha |R_j|$. (Note that for all $i\geq 1$ if both orientations work then we choose one arbitrarily.)

We stop if $|E_i| < \alpha n$.
\end{alg}

Let $m$ be such that the algorithm stops
with $|E_{m+1}| < \alpha n$. As $|M|<n$,
we have $m < 1/\alpha$.

Let $R$ be the set of reachable colours,
that is, $R=\cup_{i=1}^m R_i$.
We also write $V_{reach}=\cup_{i=1}^m V_i$.

Our next lemma shows that the first step of algorithm 
finds some linear proportion of reachable colours. 
\begin{lem} $|R_1|\geq  \alpha n$. 
\end{lem}

\begin{proof} 
First we claim that for any colour $c\in F$, 
there can be at most one $c$-edge contained in $V_0$.
Indeed, fix $c\in F$, and suppose there exist two such edges, 
say $e_1$ and $e_2$. Let $m_c=ut(u)$, where $u\in S_0$.
Fix any external $C_0$-edge $e_3$ incident to $t(u)$
(by definition of $S_0$) and suppose without loss of
generality that $e_2 \cap e_3 = \es$
(we have $e_1 \cap e_2 =\es$, as the colouring is proper).
Then $M'=M\setminus \{ut(u)\}\cup\{e_2, e_3\}$ 
is a larger rainbow matching, which contradicts
our choice of $M$, so the claim holds.

We deduce that there are at least $\eps |F| n/2$
external $F$-edges. Indeed, the total number
of $F$-edges is at least $(1+\eps)n|F|$, of which
at most $|F|$ are in $V_0$ (by the claim),
and at most $|F|n$ are contained in $V(M)$
(as $|M|<n$ and the colouring is proper).

Next we claim that at least $\eps |F| n/4$
of these external $F$-edges are good. 
It suffices to show that for each $c \in F$
there are at most $2/\alpha$ bad $c$-edges.
To see this, fix $c \in F$, let $m_c=ut(u)$,
where $u \in S_0$, and consider any
external $c$-edge $vw$ that is bad,
with $w \in V_0$. There are at least
$\alpha |C_0|$ external $C_0$-edges
incident to $t(u)$ (by definition of $S_0$),
of which at most $\alpha |C_0|/2$ 
do not contain $w$ (by definition of `bad'),
so $t(u)w$ has multiplicity
at least $\alpha |C_0|/2$ in $C_0$-edges.
As the colouring is proper, 
there are at most $|C_0|$ (external)
$C_0$-edges incident to $t(u)$,
and each bad external $c$-edge 
determines a unique $w$ incident
to at least $\alpha |C_0|/2$ 
of them, so the claim follows.

On the other hand, by definition of $R_1$
there are at most
$2|R_1| |F| + n \cdot \alpha|F|$
good external $F$-edges, so by the claim
$\eps |F| n/2 \le 2|R_1| |F| + n \cdot \alpha|F|$,
giving $|R_1|\geq \eps n/8 - \alpha n/2 \geq \alpha n$.
\end{proof}

Our next definition and associated 
lemma justify our description
of the colours in $R$ as `reachable'.
The lemma shows that they can be replaced
in the matching, maintaining its size
and being rainbow, while making only
constantly many changes and avoiding
certain proscribed other usages.

\begin{definition} \label{def:robust}
Say that $M$ is $(i,\lambda,\lambda',a)$-robust 
if for any rainbow matching $M'$
which is $\lambda$-close to $M$,
any $c\in R$ such that $m_c=vt(v)\in M'$ 
with $v\in \cup_{j=1}^i V_j$, and 
any $E_{\textit{fix}}\subseteq M'\setminus \{m_c\}$, 
any $V_{\textit{avoid}}\subseteq V\setminus  V(M')$ 
and $C_{\textit{avoid}}\subseteq C\setminus C(M')$, 
each of size at most $a$,
there is a rainbow matching $M''$ in $G$ such that 
\begin{itemize}
\item [(1)] $c\notin C(M'')$ and $v\notin V(M'')$,
\item [(2)] $M''$ is $\lambda'$-close to $M$,
\item [(3)] $M''$ fixes $E_{\textit{fix}}$,
and avoids $C_{\textit{avoid}}$ and $V_{\textit{avoid}}$.
\end{itemize}
\end{definition}

We will apply the following statement only
in the case $i=m$, but we formulate it for all $i$
to facilitate an inductive proof.
We write $f(i)=3 \cdot 2^{i-1} - 2$. 

\begin{lem}\label{lem:kickout}
Let $1 \le i \le m$ and $0 \le \lambda \le 2^{15/\eps - i/15}$. 
Then $M$ is $(i,\lambda,\lambda+f(i),2(m-i+1))$-robust.
\end{lem}

\begin{proof} 
Let $M'$, $c$, $v$, $E_{\textit{fix}}$, $V_{\textit{avoid}}$ 
and $C_{\textit{avoid}}$ be given as in Definition \ref{def:robust}.
We will show by induction on $i$ that there is a rainbow matching 
$M''$ satisfying (1), (2) and (3) of Definition \ref{def:robust}.

First we consider the base case $i=1$.  
We consider the set 
$\mc{T}$ of all pairs $(vt(v)w,ut(u)z)$
of $2$-edge paths, where $t(v)w$ 
is a good $F$-edge, $w$ and $z$ 
are distinct and in $V_0$, and we have
$c(t(v)w)=c(ut(u))$ and $c(t(u)z) \in C_0$.  If $vt(v), ut(u)$ also appear in $M'$ (a priori  they are from $M$),
we will obtain $M''$ from $M'$ 
by removing $vt(v)$ and $ut(u)$ 
and adding $t(v)w$ and $t(u)z$.
Note that any such $M''$ is 
$(\lambda+4)$-close to $M$, 
as required for (2).
There are various constraints that
need to be satisfied for $M''$ to be
a rainbow matching satisfying (1) and (3);
we will see below that $\mc{T}$ is large 
enough so that we can pick an element of $\mc{T}$ such that the matching derived from it satisfies these constraints.

We claim that $|\mc{T}| \ge \alpha^3 |C_0|n/2$.
To see this, note that as $v \in V_1$ 
there are at least $\alpha |F|$
good external $F$-edges incident to $t(v)$.
For each such edge $t(v)w$, there is an edge 
in $M$ of the same colour, which we can
label as $ut(u)$ so that there are at least
$\alpha |C_0|/2$ external $C_0$-edges $t(u)z$
with $z \ne w$. As $|F| \ge \alpha n$ 
by Lemma \ref{lem:f} the claim follows.

Next we estimate how many configurations
in $\mc{T}$ are forbidden by the various
constraints described above.

We start by considering the constraints on $w$,
namely $w \in V'_0$ and $w \notin V_{\textit{avoid}}$.
These forbid at most $2\lambda + 2m$ choices of $w$.
For each such $w$, there are at most $n/k$
choices of an $F$-edge $t(v)w$, which determines
$ut(u) \in M$ with $c(ut(u))=c(t(v)w)$,
and then there are at most $|C_0|$ choices
for a $C_0$-edge $t(u)z$. Thus we forbid at most
$(2\lambda + 2m) |C_0| n/k$ configurations.

We also have the same constraints on $z$,
i.e.\ $z \in V'_0$ and $z \notin V_{\textit{avoid}}$.
Again, this forbids at most $2\lambda + 2m$ choices of $z$.
For each such $z$, there are at most $|C_0|$ choices
for a $C_0$-edge $t(u)z$, which determines $ut(u) \in M$,
and then at most one edge $t(v)w$ with the same colour.
Thus we forbid at most $(2\lambda + 2m)|C_0|$ configurations.

Now we consider the constraints on $ut(u)$,
namely $ut(u) \in M'$ and $ut(u) \notin E_{\textit{fix}}$.
These forbid at most $\lambda + 2m$ choices of $ut(u)$.
For each such $ut(u)$, there is at most
one edge $t(v)w$ with the same colour,
and at most $|C_0|$ choices for $t(u)z$,
so we forbid at most $(\lambda + 2m)|C_0|$ configurations.

For $t(v)w$ we have the constraint
$c(t(v)w) \in C(M')$, which forbids at most
$\lambda$ edges $t(v)w$. This fixes $ut(u) \in M$
with the same colour, and then there are
at most $|C_0|$ choices for $t(u)z$,
so we forbid at most $\lambda|C_0|$ configurations.

Finally, for $t(u)z$ we have the constraint 
$c(t(u)z) \notin C_{\textit{avoid}} \cup (C(M') \sm C(M))$.
There are at most $2m+\lambda$ choices 
for $c' \in C_{\textit{avoid}}\cup (C(M') \sm C(M))$,
then at most $n$ choices for a $c'$-edge $t(u)z$,
which determines $ut(u)$ and then at most one
edge $t(v)w$ with the same colour,
so we forbid at most $(2m+\lambda)n$ configurations.

The total number of forbidden configurations
is at most $2(\lambda + m) |C_0|(2 + n/k)
+ (2m+\lambda)n < \alpha^3 |C_0| n/2 \le |\mc{T}|$,
as $\lambda/k \le 2^{-5/\eps} < 10^{-6} \eps^3$,
so $2(\lambda + m) |C_0|(2 + n/k) < \alpha^3 |C_0| n/4$
and $\alpha^3 |C_0| n/4 > \alpha^3 kn/4
> (2m+\lambda)n$. We can therefore
pick a configuration in $\mc{T}$
that satisfies all the above constraints,
which completes the proof of the base case.

Now suppose $i>1$.
As $v\in V_i$, there is $j<i$ such that
the number of $R_j$-edges $t(v) w$
with $w \in \cup_{j=0}^{i-1} V_{j}$
is at least $\alpha |R_j| \ge \alpha^2 n$.
We will consider two cases according to whether
most of these edges go to $V_0$. In each
case we note by the induction hypothesis
that $M$ is $(j,\lambda,\lambda+f(j),2(m-j+1))$-robust.

Suppose first that at least $\alpha^2 n/2$
of these $R_j$-edges $t(v) w$ have $w \in V_0$.
Then we can fix such an edge $t(v)w$
with $c'=c(t(v)w)\in R_j$ such that
$w \in V_0'\setminus V_{\textit{avoid}}$
and $m_{c'}\in M'\setminus E_{\textit{fix}}$
(the first condition forbids at most 
$(2m+2\lambda) n/k$ choices and
the second at most $2m+\lambda$ choices).

By Definition \ref{def:robust}
applied to $m_{c'}$ in $M'$ with sets 
$E_{\textit{fix}}':=E_{\textit{fix}}\cup\{vt(v)\}$,  
$V_{\textit{avoid}}':=V_{\textit{avoid}}\cup \{w\}$ 
and $C_{\textit{avoid}}':=C_{\textit{avoid}}$,
each of size at most $2(m-i+1)+1 < 2(m-j+1)$,
we obtain a rainbow matching $M_1$ 
which is $(\lambda+f(j))$-close to $M$, 
satisfies (3), contains $v(t)v$,
and avoids the colour $c'$ and vertex $w$.
As $\lambda+f(j)+2 \le \lambda+f(i)$, we see that
$M''= M_1\setminus \{vt(v)\}\cup \{t(v)w\}$
is as required to complete the proof in this case.

It remains to consider the case that
at least $\alpha^2 n/2$ of the $R_j$-edges 
$t(v)u$ have $u \in \cup_{j=1}^{i-1} V_{j}$.
We can fix such an edge $t(v)u$
with $c(t(v)u)=c_1 \in R_j$ such that 
$c(ut(u)) = c_2 \in R_{j'}$ with $1 \le j' \le i-1$
and $ut(u) \in M' \sm E_{\textit{fix}}$
(this forbids at most $(2m+\lambda) n/k$ choices),
and $m_{c_1} \in M' \sm E_{\textit{fix}}$
(which forbids at most $2m+\lambda$ choices).

By Definition \ref{def:robust}
applied to $m_{c_1}$ in $M'$ with sets 
$E_{\textit{fix}}':=E_{\textit{fix}}\cup\{vt(v),ut(u)\}$,  
$V_{\textit{avoid}}':=V_{\textit{avoid}}$ 
and $C_{\textit{avoid}}':=C_{\textit{avoid}}$,
each of size at most $2(m-i+1)+2 \le 2(m-j+1)$,
we obtain a rainbow matching $M_1$ 
which is $(\lambda+f(j))$-close to $M$, 
satisfies (3), contains $v(t)v$ and $ut(u)$,
and avoids the colour $c_1$.

We apply Definition \ref{def:robust} 
again to $m_{c_2} = ut(u) \in M_1$ with sets
$E_{\textit{fix}}'':=E_{\textit{fix}}\cup\{vt(v)\}$,  
$V_{\textit{avoid}}'':=V_{\textit{avoid}}$ 
and $C_{\textit{avoid}}'':=C_{\textit{avoid}} \cup \{c_1\}$,
noting that $\lambda+f(j) \le 2^{15/\eps - i/15}
+ 2^{1 + 1/\alpha} \le 2^{15/\eps - j'/15}$, 
as $\alpha=\eps/12$, $j' \le i-1$,
$i \le m < 1/\alpha$ and $\eps<10^{-3}$.
Thus we obtain a rainbow matching $M_2$ 
which is $(\lambda+f(j)+f(j'))$-close to $M$, 
satisfies (3), contains $v(t)v$,
and avoids the colour $c_1$ and the vertex $u$.
As $\lambda+f(j)+f(j')+2 \le \lambda+f(i)$, we see that
$M''= M_2\setminus \{vt(v)\}\cup \{t(v)u\}$
is as required to complete the proof.
\end{proof}

\section{Structure and counting}
\label{sec:mainproof}

Continuing with the proof strategy outlined above,
we now use the maximality assumption on $M$
to prove the following structural properties of $G$,
which we will then combine with a counting argument
to obtain a contradiction to the assumption that $|C_0|>k$,
thus completing the proof of Theorem~\ref{thm:main}.

In the proof of the following lemma
we repeatedly use Lemma \ref{lem:kickout}
in the case $i=m$, i.e.\ that
$M$ is $(m,\lambda,\lambda+f(m),2)$-robust,
for any $0 \le \lambda \le 2^{15/\eps - m/15}$. Here we recall that $R=\cup_{i=1}^m R_i$ 
and $m < 1/\alpha = 12/\eps$,
so Definition \ref{def:robust} is applicable
whenever $0 \le \lambda \le 2^{14/\eps}$.

\begin{lem}\label{ri-properties} \
\begin{itemize} 
\item[(C1)]  Any $v\in V_{\textit{reach}}$ is 
not incident to an external $R$-edge, 
\item [(C2)]  Any two $u,v\in V_{\textit{reach}}$  
are not adjacent by an $R$-edge.
\item [(C3)] There are no $R$-edges contained in $V_0$.
\end{itemize}
\end{lem}

\begin{proof} 
\textbf{(C1):} Suppose $v$  is incident to an external $R$-edge, 
say $w\in V_0$ and $c=c(vw)\in R$.  
Applying Definition \ref{def:robust} to $vt(v)$ in $M'=M$  
with sets $E_{\textit{fix}}=\{m_c\}$ and
$V_{\textit{avoid}}=\{w\}$, $C_{\textit{avoid}}=\es$,
we obtain a rainbow matching $M_1$
that is $f(m)$-close to $M$, contains $m_c$ and avoids $w$.
As $f(m) < 2^{14/\eps}$, 
we can apply Definition \ref{def:robust} 
again to $m_c$ in $M_1$ with sets 
$E_{\textit{fix}}=C_{\textit{avoid}}=\es$ and
$V_{\textit{avoid}}=\{v,w\}$,
obtaining a rainbow matching $M_2$
that is $2f(m)$-close to $M$, and avoids 
the vertices $v$, $w$ and the colour $c$.
Now $M_2 \cup\{uw\}$ is a larger rainbow matching 
than $M$, which is a contradiction. 

\textbf{(C2):}
Suppose for a contradiction
that we have $u,v\in V_{\textit{reach}}$ 
and an edge $uv$ with $c(uv) = c_3 \in R$.
Let $c_1=c(vt(v))$ and $c_2=c(ut(u))$.
Then $c_1$ and $c_2$ are also in $R$.

Applying Definition \ref{def:robust} to $vt(v)$ in $M'=M$  
with sets $E_{\textit{fix}}=\{ut(u),m_{c_3}\}$ and
$V_{\textit{avoid}}=C_{\textit{avoid}}=\es$,
we obtain a rainbow matching $M_1$
that is $f(m)$-close to $M$, 
contains $ut(u)$ and $m_{c_3}$,
and avoids $v$.

Applying Definition \ref{def:robust} again to $ut(u)$ in $M_1$
with sets $E_{\textit{fix}}=\{m_{c_3}\}$,
$V_{\textit{avoid}}=\{v\}$ and
$C_{\textit{avoid}}=\es$,
we obtain a rainbow matching $M_2$
that is $2f(m)$-close to $M$, contains $m_{c_3}$,
and avoids $u$ and $v$.

Applying Definition \ref{def:robust} 
a final time to $m_{c_3}$ in $M_2$
with $V_{\textit{avoid}}=\{u,v\}$ and
$E_{\textit{fix}}=C_{\textit{avoid}}=\es$,
noting that $2f(m) < 2^{14/\eps}$, 
we obtain a rainbow matching $M_3$
that is $3f(m)$-close to $M$, 
and avoids the vertices $u$ and $v$
and the colour $c_3$.

Now $M_3 \cup\{uv\}$ is a larger rainbow matching 
than $M$, which is a contradiction. 

\textbf{(C3):} 
We omit the proof as it is very similar 
to that of the first two statements.
\end{proof}

\begin{proof}[Proof of Theorem~\ref{thm:main}]
Recall that $M$ is a maximum rainbow matching in $G$,
we write $V_0 = V(G) \sm V(M)$ and suppose
for a contradiction that $|M|<n-k$.
Recall also that the reachable colours
are $R=\cup_{i=1}^m R_i$,
we write $V_{reach}=\cup_{i=1}^m V_i$,
and the algorithm stopped 
with $|V_{m+1} |< \alpha n$.

As every colour appears at least $(1+\eps)n$ times
the number of $R$-edges is at least $|R| (1+\eps)n$.
By Lemma~\ref{ri-properties} they are all incident
to $T(V_{reach})$ or $V(M) \sm V_{reach}$.
Write \[V^* = T(V_{reach}) \cup V_{m+1}\cup T(V_{m+1})
\ \text{ and } \
V' = V(M) \sm (V_{reach} \cup V^*). \]
At most $|V^*||R| < (|R|+2\alpha n)|R|$ 
of the $R$-edges are incident to $V^*$, so at least 
$|R| (1+\eps)n - (|R|+2\alpha n)|R| > |R|(|V'|+ \eps n)/2$ 
are incident to $V'$ but not to $V^*$.

By definition of our algorithm, 
for each $v$ in $V'$ the number of $R$-edges
$vu$ with $u \in V_0 \cup V_{reach}$
is at most $\alpha |R|$.
All remaining $R$-edges are contained in $V'$,
and they number at least
$|R|(|V'|+ \eps n)/2 - \alpha |R||V'| > |R||V'|/2$.
But the colouring is proper, 
so there are at most $|V'|/2$
edges of any colour contained in $V'$.
This contradiction completes the proof.
\end{proof}

\section{Conclusion}
A similar result to Theorem \ref{thm:main}
was very recently obtained independently by
Gao, Ramadurai, Wanless and Wormald \cite{grww}.
Their result is closer than ours 
to the spirit of Conjecture \ref{conj:ab},
as they obtain a full rainbow matching
(whereas we allow a constant number of unused colours).
However, our results are incomparable, 
as they need a stronger bound on edge multiplicities,
namely $\sqrt{n}/\log^2 n$.
Our algorithm is deterministic, whereas theirs is randomised,
and the analysis of our algorithm is simpler.
In any case, the two approaches are very different,
so we think that it will be valuable to pursue both sets of ideas,
and perhaps also the approach of Pokrovskiy,
in making further progress on Conjecture \ref{conj:ab}.

\end{document}